\documentclass[10pt, twocolumn, a4paper]{article}
\usepackage{hyperref}
\usepackage{graphicx}
\usepackage{amssymb}
\usepackage{amsmath}
\usepackage{amsthm}
\usepackage{lineno}
\usepackage{subcaption}
\hypersetup{
  unicode=true,
  pdftitle={Box counting dimensions of generalised fractal nests},
  pdfauthor={Sinisa Milicic},
  colorlinks=true,
  linkcolor=magenta,
  citecolor=blue,
  filecolor=magenta,
  urlcolor=cyan 
}

\allowdisplaybreaks

\def\R{{\mathbb R}}
\def\Sn1{{\mathbb S}^{n-1}}
\def\S1{{\mathbb S}^1}
\def\N{{\mathbb N}}
\def\C{{\mathcal C}}
\def\F{{\mathcal F}}
\def\O{{\mathcal O}}
\def\Mdl{{\mathcal M_{*d}}}
\def\Mdu{{\mathcal M^*_d}}
\def\Md{{\mathcal M_d}}
\def\mo#1{{\mathrm{#1}}}
\newtheorem{theorem}{Theorem}
\newtheorem{lemma}{Lemma}
\newtheorem{proposition}{Proposition}
\newtheorem{corollary}{Corollary}
\newtheorem{definition}{Definition}
\newtheorem{example}{Example}
\numberwithin{equation}{section}
\begin{document}

\title{Box counting dimensions of generalised fractal nests}

\author{Sini\v sa Mili\v ci\'c\footnote{smilicic@unipu.hr}\\
  \small{Juraj Dobrila University of Pula}\\ \small{Faculty of
    Informatics}\\ \small{52 100 Pula, Croatia}}

\maketitle

\begin{abstract}
  Fractal nests are sets defined as unions of unit $n$-spheres scaled
  by a sequence of $k^{-\alpha}$ for some $\alpha>0$. In this article
  we generalise the concept to subsets of such spheres and find the
  formulas for their box counting dimensions. We introduce some novel
  classes of parameterised fractal nests and apply these results to
  compute the dimensions with respect to these parameters. We also
  show that these dimensions can be seen numerically. These results
  motivate further research that may explain the unintuitive behaviour
  of box counting dimensions for nest-type fractals, and in general
  the class of sets where the box-counting dimension differs from the
  Hausdorff dimension.
\end{abstract}

\section{Introduction}

Research into rectifiability (eg.\ \cite{tricot}) has given some
unexpected results in the differences between the Hausdorff dimension
\cite[pg.\ 171]{federer} and the box counting dimension of some
countable unions of sets and their smooth generalisations \cite[pg.\
121-122]{tricot}, such as unrectifiable spirals. Recently, progress
has been made in the application of the box counting dimension in the
analysis of complex zeta functions, \cite{lazura}. We note especially
the discovery of various interesting properties, including a
relationship of Lapidus-style zeta functions with the Riemann zeta
function and the generalisation of the concept of the box dimension
to complex dimensions. Fractal nests, as presented and analysed in
\cite{lazura}, behave in an unexpected way with respect to the
appropriate exponents. There are two general types of behaviour of
fractal nests. The well understood type of behaviour of fractal nests
concerns sets that locally resemble Cartesian products of fractals
\cite[pg.~227]{lazura}, \cite[Remark 6.]{zuzu}. In such sets, the
dimension behaves naturally as the sum of the dimensions of ``base''
sets. What remains less understood are dimensions of fractal nests of
centre type. In this case, the dimension is product-like in terms of
dimensions of underlying elements, lacking an intuitive geometric
interpretation.

This article focuses on a more classical approach to the box-counting
dimension, giving examples that may further the understanding of how
the box counting dimension behaves with respect to countable unions of
similar sets. Whereas in \cite{lazura} and \cite{zuzu} the fractal
nests studied are based on $(n-1)$-spheres (hyper-spheres), we study
fractal nests $S_\alpha$ based on fractal subsets of box counting
dimension $\delta$ of such spheres under similar scaling. Our results
are compatible with the cited ones, having them as limit cases of full
dimension, $\delta=(n-1)$; notably, the dimension are
$$\dim_B S_\alpha = \frac{\delta+1}{\alpha+1}$$ for the centre type
and $$\dim_B S_\alpha = \delta +\frac1{\alpha+1}$$ for the outer
type. The proofs of these results hint at some more general geometric
and topological properties.

This paper is divided into five sections. After this introduction we
present the main concepts and results of this paper -- a dimensional
analysis of generalised fractal nests, followed by applications of
these results to novel fractal sets with numerical examples. After
that, we provide the proofs of our main results and in the final
section we remark on some issues and problems that naturally arise
from these investigations.

\section{Generalised fractal nests}

\subsection{Box counting dimension}

We start with the definition of the box counting dimension as stated
in \cite[pg.\ 28]{falconer} as the alternative definition. By an
$\epsilon$-mesh in $\R^n$ we understand the partition of $\R^n$ into
disjoint (except possibly at the border) $n$-cubes of side $\epsilon$.

\begin{definition}
  Let $S\subseteq\R^n$ be a bounded Borel set. By $N_{\epsilon}(S)$ we
  denote the number of such $n$-cubes of the $\epsilon$-mesh of $\R^n$
  that intersect $S$. We define the {\bf upper (lower) box counting
    dimension} of $S$ by
  \begin{align*}
    \overline{\dim}_BS&=\inf\left\{\delta\,\Big|\,\epsilon^{\delta}N_\epsilon(S)\to0\right\}.\\
    \bigg(\underline\dim_BS&=\sup\left\{\delta\,\Big|\,\epsilon^{\delta}N_\epsilon(S)\to+\infty\right\}\bigg).
  \end{align*}
\end{definition}

\begin{example}\label{mkocka}
  The unit $m$-cube $K_m=[0,1]^m\times\{0\}^{n-m}$ in $\R^n$ with $m\leq n$
  has $$\overline{\dim}_B K_m = \underline\dim_B K_m = m$$ because it
  intersects $\left\lceil\epsilon^{-1}\right\rceil^m$ $n$-cubes of side
  $\epsilon$. Hence, for various $\delta$ we have, as $\epsilon\to0$,
  $$\epsilon^\delta\left\lceil\epsilon^{-1}\right\rceil^m\to
  \begin{cases} 0, \hbox{ for }\delta>m,\\
    1, \hbox{ for }\delta=m,\\
    +\infty,\hbox{ for }\delta<m.\end{cases}$$
\end{example}

The box counting dimension of a set $S$ can be defined in terms of
other counting functions, such as the maximum number of disjoint
$\epsilon$-balls centred on points of $S$, the minimal number of
$\epsilon$-balls needed to cover $S$, similar constructions in
equivalent metrics, etc. \cite[pg.\ 30]{falconer}.

One important reformulation of the box counting dimension is the
Min\-kow\-ski-Bouligand dimension. It is formulated in terms of the
Lebesgue measure in the ambient space and constructs a natural
``contents'' function at every dimension, in that regard similar to
the Hausdorff dimension and measure.

As $\epsilon$-balls used here and in other literature correspond to
the Euclidean metric, we need to compensate for the coefficient for
volume of the ball characteristic to this metric and
dimension, $$\gamma_{x}=\frac{\pi^{\frac x2}}{\Gamma(\frac x2+1)}.$$
For further discussion of this coefficient and its use in the
Min\-kow\-ski-Bouligand dimension, see \cite{maja} and
\cite[Chapter 3.3]{krantzparks}.

\begin{definition}\label{mbdim}
  Let $S\subseteq \R^n$. For $\epsilon>0$, we define the {\bf
    $\epsilon$-Min\-kow\-ski sausage of $S$} as the set
  $$(S)_\epsilon = \left\{ x \in\R^n \,|\, d(x,S)\leq \epsilon\right\}.$$

  Let $\lambda$ be the Lebesgue measure on $\R^n$ and denote by $A^{n,d}_\epsilon(S)$ the ratio
  $$A^{n,\delta}_\epsilon(S):=\frac{\lambda (S)_\epsilon}{\gamma_{n-\delta}\epsilon^{n-\delta}}.$$

  We say that $d$ is the {\bf upper (lower) Min\-kow\-ski-Bou\-li\-gand dimension of
    $S$}, $\overline{\dim}_{MB}S$ ($\underline\dim_{MB}S$) if
  \begin{align*}
    d&=\inf \left\{\delta\,\Big|\,\lim_{\epsilon\to0} A_\epsilon^{n,\delta} (S)= +\infty\right\}\\
    \bigg(d&=\sup\left\{\delta\,\Big|\,\lim_{\epsilon\to0} A_\epsilon^{n,\delta} (S)= 0\right\}\bigg).
  \end{align*}
\end{definition}

A classical result (eg.\ \cite[Ch.\ I.2]{pesin}, \cite[Prop.\ 
2.4]{falconer}) is that both upper and lower Min\-kow\-ski-Bou\-li\-gand
dimensions are exactly equal to the corresponding upper and lower box
counting dimensions, so we will only use the box-dimension notation
for the discussed dimension.

\begin{definition}
  For $S\subseteq\R^n$, if $\underline\dim_BS=\overline{\dim}_BS=d$,
  we say that $S$ is of box counting dimension $d$, denoting $$\dim_BS=d.$$

  For $S$ of box-dimension $d$, we define the {\bf
    normalised upper (lower) Minkowski content of $S$} as
  \begin{align*}
    \Mdu (S) &= \limsup_{\epsilon\to0} A_\epsilon^{n,d}(S),\\
    \bigg(\Mdl (S) &= \liminf_{\epsilon\to0} A_\epsilon^{n,d}(S))\bigg).
  \end{align*}

  We say that the set $S\subseteq \R^n$ of box-dimension $d$ is
  {\bf Min\-kow\-ski-non-degenerate} if $\Mdu(S) <+\infty$ and
  $\Mdl(S) > 0$ and denote $$\dim_BS\equiv d.$$
\end{definition}

\begin{example}
  As per the earlier discussion in Example \ref{mkocka}, it is easy to
  show that $K_m\subseteq \R^n$ is of normalised Minkowski content
  (both upper and lower) equal to $1$ because
  $$\frac{\lambda(K_m)_\epsilon}{\epsilon^{n-m}}=\gamma_{n-m}+\sum_{k=0}^{m-1}\gamma_{n-k}\left(\frac\epsilon2\right)^{m-k}\!E^{(k)}_m,$$
  where $E_m^{(k)}$ is the number of $k$-edges of the $m$-cube.

  \begin{figure}[t]
    \centering
    \includegraphics[width=0.3\textwidth]{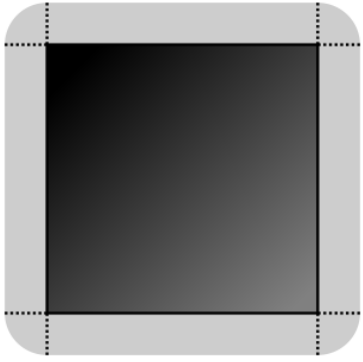}
    \caption{The Minkowski sausage of the unit square in $\R^2$.}
    \label{fig:square}
  \end{figure}

  In Figure \ref{fig:square}, we have $E_2^{(0)}=4$ and $E_2^{(1)}=4$
  and $\gamma_0=1$, $\gamma_1=2$ and $\gamma_2=\pi$, so we have
  $\lambda(K_2)_\epsilon =
  \gamma_0+\frac\epsilon2\gamma_1E_2^{(1)}+\frac{\epsilon^2}4\gamma_2E_2^{(0)}=1+4\epsilon+\pi\epsilon^2$.
\end{example}

\begin{example}\label{ealfa}
  The set of points $E_\alpha=\{k^{-\alpha}\,|\,k\in\N\}$ for $\alpha>0$
  is of box counting dimension $\frac1{1+\alpha}$, with normalised
  Minkowski contents (both upper and lower) equal to
  $$\Md S=\left(\frac2{\alpha\sqrt\pi}\right)^{\frac\alpha{\alpha+1}}\!(\alpha+1)\,\Gamma\left(\frac{\alpha}{2(\alpha+1)}+1\right).$$

  \begin{figure}[ht]
    \centering
    \includegraphics[width=0.45\textwidth]{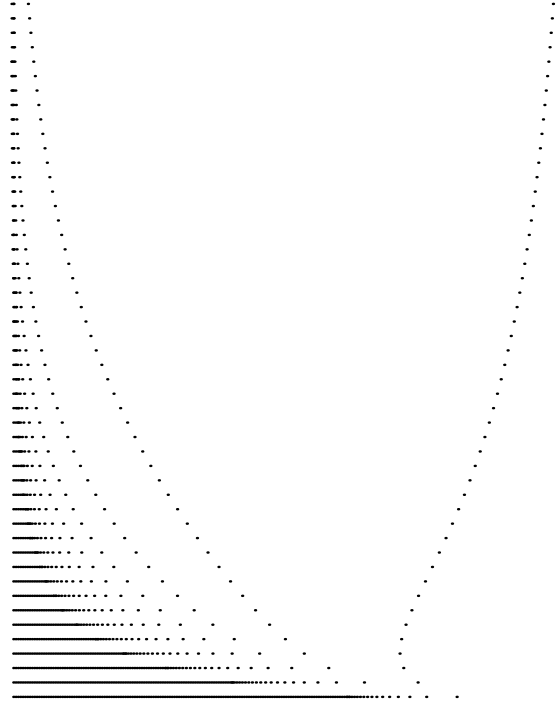}
    \caption{Sets $E_\alpha$ of unit normalised Minkowski contents for
      $\alpha\in\{0.2,0.4,\ldots,5.0\}$.}
    \label{fig:ealpha}
  \end{figure}

  Figure \ref{fig:ealpha} shows sets $E_\alpha$ of unit normalised
  Min\-kow\-ski contents in dimensions ranging from $5/6$ at the bottom to
  $1/6$ in the top row.
\end{example}

\subsection{Fractal analysis of $\alpha$-regular generalised fractal nests}

Intuitively, an inner $\alpha$-regular fractal nest is the subset of
the union of circles of radii $n^{-\alpha}$ where subsets per
individual circle are homothetic to each other. The outer
$\alpha$-regular fractal nest has radii of form $1-(k+1)^{-\alpha}$.

For a set $S\subseteq\R^n$ and $x\geq0$ we denote by $(x)S$ the
scaling of the set $S$ by $x$.

\begin{definition}Let $S\subseteq \Sn1$ be a Borel subset of the
  unit sphere in $\R^n$. We define the {\bf $\alpha$-regular fractal nest of centre (outer) type} as
  \begin{align*}
    \F_\alpha S&=\bigcup_{k=1}^\infty (k^{-\alpha})S\\
    \bigg(\O_\alpha S&=\bigcup_{k=1}^\infty (1-k^{-\alpha})S\bigg).
  \end{align*}
\end{definition}

Note that $$\F_\alpha\{1\}=E_\alpha$$ and
$$\O_\alpha\{1\} = 1-E_\alpha,$$ where for a set $S$ the expression
$1-S$ is defined as $\{1-x\,|\,x\in S\}$, with $E_\alpha$ defined in Example \ref{ealfa}.

\begin{theorem}\label{th:inner}
  Let $S\subseteq \Sn1$ be a Borel subset of the unit hyper-sphere in
  $\R^n$ such that $$\dim_BS\equiv \delta\geq 0.$$ For every
  $\alpha>0$ the $\alpha$-regular fractal nest of centre
  type generated by $S$ has
  $$\dim_B \F_\alpha(S)
  \begin{cases}
    \equiv \frac{\delta+1}{\alpha+1},\hbox{ for } \alpha\delta < 1,\\
    = \delta,\hbox{ for } \alpha\delta=1\\
    \equiv \delta,\hbox{ for } \alpha\delta > 1.
  \end{cases}
  $$
\end{theorem}

\begin{theorem}\label{th:outer}
  Let $S\subseteq \Sn1$ be a Borel subset of the unit hyper-sphere
  in $\R^n$ such that $$\dim_BS\equiv \delta\geq 0.$$ For every
  $\alpha>0$ that the $\alpha$-regular fractal nest of outer
  type generated by $S$ has
  $$\dim_B \O_\alpha S\equiv \delta+\frac1{\alpha+1}.$$
\end{theorem}

The proofs of both theorems rely on the well-known technique of
separating the ``core'' and the ``tail'' of the set, the ``tail'' part
consisting of the well-isolated components, and the ``core'' of the
remaining parts.

We take a novel approach to analysing the ``core'', where we use the
covering lemma (Lemma \ref{coverlemma}) to replace the components of
the core with well-spaced sets without losing the asymptotic and hence
dimensional properties, including Minkowski (non)-degeneracy.

\section{Application of the generalised nest formulas}

In this section, we show some applications of our main results to some
known and some novel fractal sets.

\subsection{Mapping $m$-cubes to $m$-spheres}

Let $K_m\subset\R^n$ be a unit $m$-cube as defined in Example
\ref{mkocka} with $m<n$. It is easy to show that for limited domains
such as a unit $m$-cube, the mapping
\begin{align*}
  \Phi_{n-1}\colon(x_1,\ldots,x_{n-1})\mapsto(&\cos x_1,\sin x_1\cos x_2,\ldots,\\
                                              &\sin x_1\sin x_2\cdots\sin x_{n-1})
\end{align*} is bi-Lipschitz.

Thus, any subset of the unit hyper-cube can be mapped to a corresponding set
of equal dimension on the hyper-sphere. If we identify $K_m$ with its
embedding into $\Sn1$, applying Theorem \ref{th:inner} we have
that $$\dim_B\F_\alpha K_m = \frac{m+1}{\alpha+1}.$$ We note that this
formula also applies to $m=0$. Also, for the outer nests, using
Theorem \ref{th:outer} we have
$$\dim_B\O_\alpha K_m = m+\frac{1}{\alpha+1},$$ again, compatible with $m=0$.

\subsection{$(\alpha,\beta)$-bi-fractals}

Let $\alpha, \beta>0$. We can identify $E_\beta$ with its image on the
unit circle of $\Phi_1$. Let $D_\beta$ be the union
$$D_\beta=\Phi_1\left(\frac\pi4(1- E_\beta)\right)\cup\Phi_1\left(\frac\pi4 (1+E_\beta)\right)\subseteq \S1.$$

We have that $$\dim_BD_\beta = \frac1{1+\beta}.$$ We call the
set $\F_\alpha D_\beta$ an \emph{$(\alpha,\beta)$-bi-fractal}. Its dimension
is given by Theorem \ref{th:inner} as
$$\dim_B\F_\alpha D_\beta = \frac{\beta+2}{(\beta+1)(\alpha+1)}.$$

\begin{figure}[t]
  \centering
  \parindent0pt
  \begin{subfigure}[b]{0.24\textwidth}
    \includegraphics[width=\textwidth]{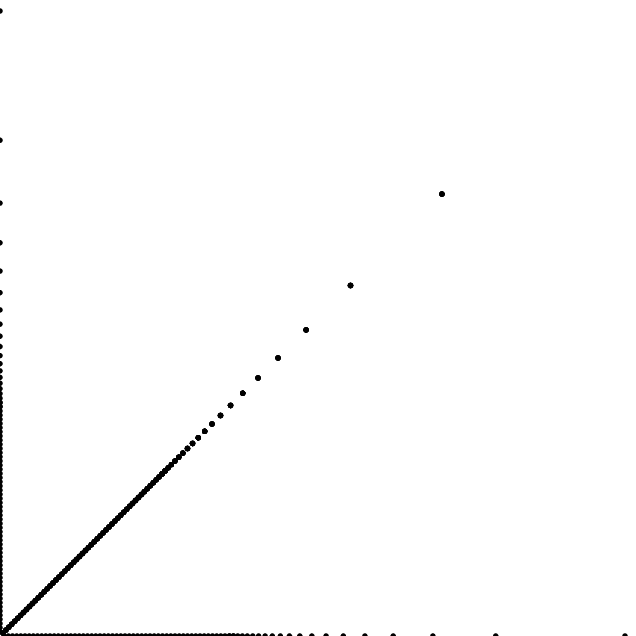}
    \caption{$\alpha\approx0.35$}
  \end{subfigure}
  \begin{subfigure}[b]{0.24\textwidth}
    \includegraphics[width=\textwidth]{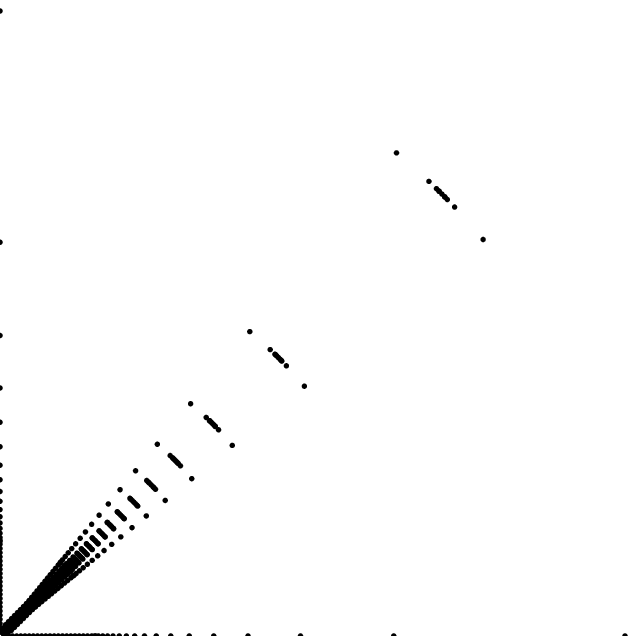}
    \caption{$\alpha\approx0.67$}
  \end{subfigure}\\
  \begin{subfigure}[b]{0.24\textwidth}
    \includegraphics[width=\textwidth]{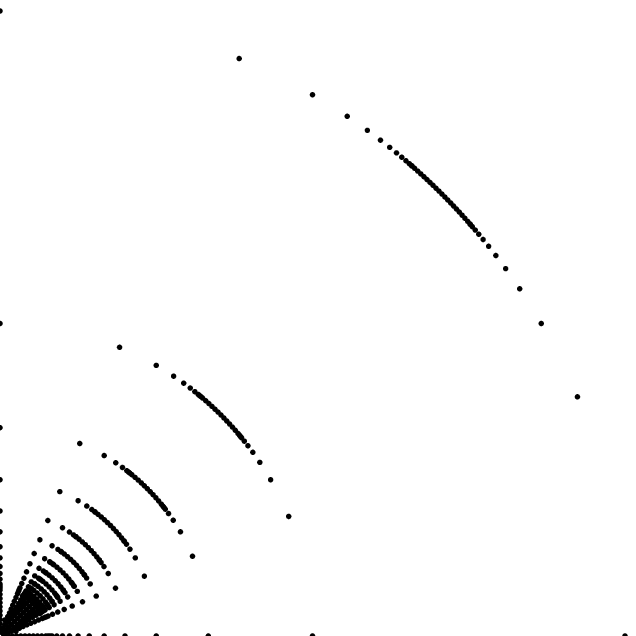}
    \caption{$\alpha\approx1.00$ }
  \end{subfigure}
  \begin{subfigure}[b]{0.24\textwidth}
    \includegraphics[width=\textwidth]{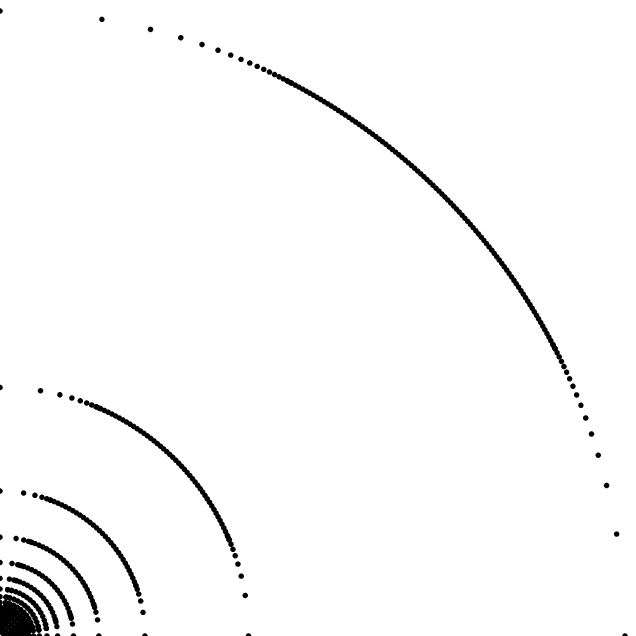}
    \caption{$\alpha\approx1.33$}
  \end{subfigure}
  \caption{$(\alpha,\beta)$-bi-fractals of dimension $3/4$}
  \label{fig:BF}
\end{figure}
\begin{figure}[t]
  \centering
  \parindent0pt
  \begin{subfigure}[b]{0.24\textwidth}
    \includegraphics[width=\textwidth]{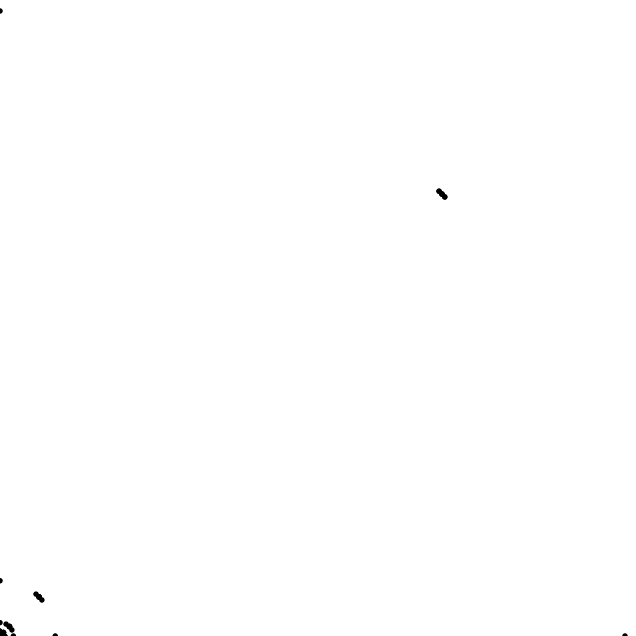}
    \caption{$\dim_B\F_\alpha D_\beta\approx 0.25$}
  \end{subfigure}
  \begin{subfigure}[b]{0.24\textwidth}
    \includegraphics[width=\textwidth]{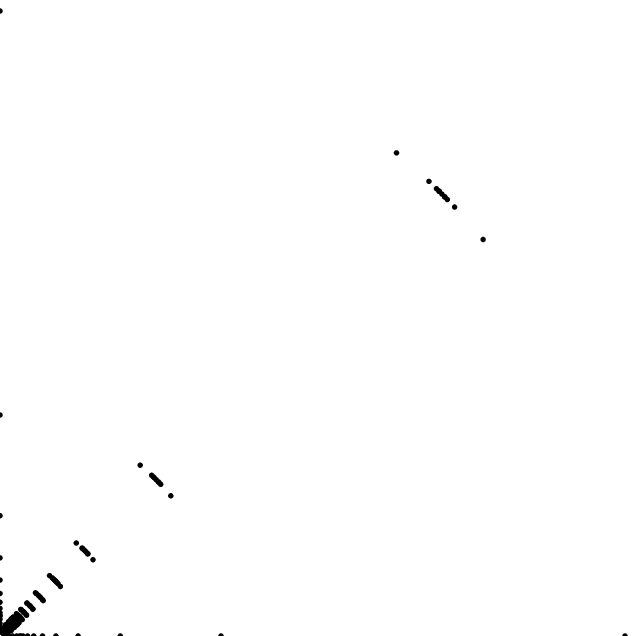}
    \caption{$\dim_B\F_\alpha D_\beta\approx 0.5$}
  \end{subfigure}\\
  \begin{subfigure}[b]{0.24\textwidth}
    \includegraphics[width=\textwidth]{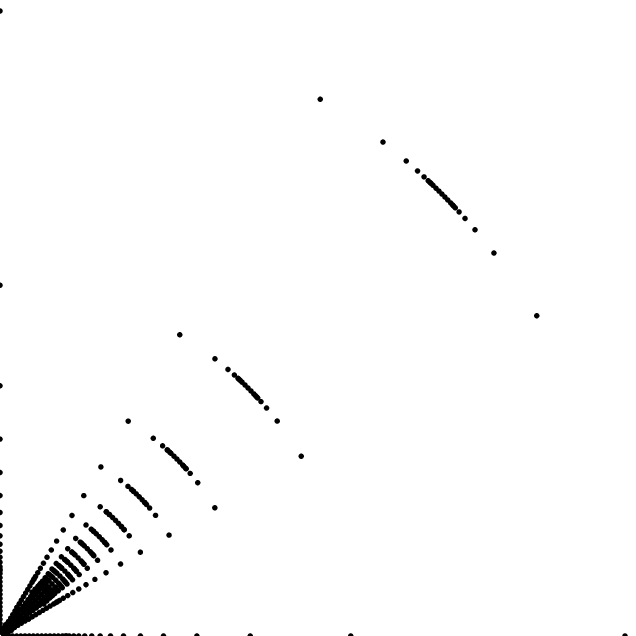}
    \caption{$\dim_B\F_\alpha D_\beta\approx 0.75$}
  \end{subfigure}
  \begin{subfigure}[b]{0.24\textwidth}
    \includegraphics[width=\textwidth]{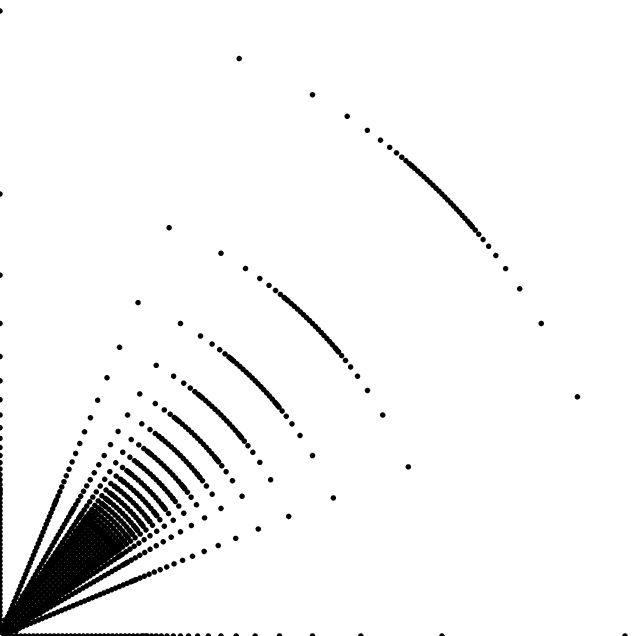}
    \caption{$\dim_B\F_\alpha D_\beta\approx 1.00$}
  \end{subfigure}
  \caption{$(\alpha,\beta)$-bi-fractals of various dimensions.}
  \label{fig:BV}
\end{figure}

\subsection{Uniform Cantor nests}

Intuitively, the uniform Cantor set $\C_N^C$ are Cantor sets
``preserving'' $N$ copies of themselves totalling $C$ in relative
length in each iteration.

In \cite[pg.\ 71]{falconer} uniform Cantor sets are defined in terms
of the number of preserved copies $m$ and the relative gap $r$ (see
\ref{fig:cantor}). Modifying that definition to our notation, uniform
Cantor sets are defined as follows.

\begin{definition}[Uniform Cantor set, \cite{falconer}]
  Let $N\geq 2$ be an integer and $0<r<\frac1N$. We define the set
  $\C_n^r$ as the set obtained by the construction in which each basic
  interval $I$ is replaced by $N$ equally spaced sub-intervals of
  lengths $r|I|$, the ends of $I$ coinciding with the ends of the
  extreme sub-intervals. The starting interval for $\C_N^r$ is $[0,1]$.
\end{definition}

\begin{figure}[ht]
  \centering
  \includegraphics[width=0.45\textwidth]{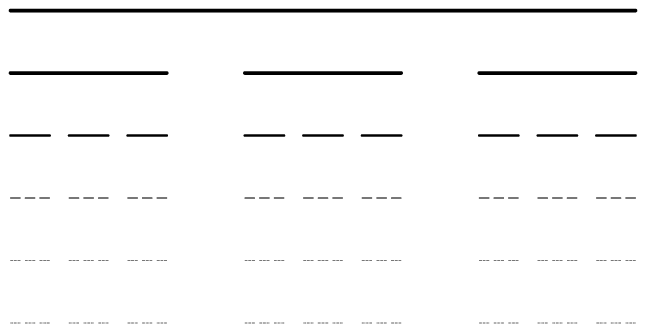}
  \caption{First four iterations of the uniform Cantor set  $\C_3^{1/4}$.}
  \label{fig:cantor}
\end{figure}
For the standard Cantor set, $\C_2^{\frac13}$ we
have, as is shown in \cite{lapidus:cantor} and \cite{jiangchen},
$$\dim_B\C_2^{\frac13}\equiv \frac{\log2}{\log3},$$
with different upper and lower Minkowski contents,
$$\Mdu\C_2^{\frac13}=\gamma^{-1}_{1-\log_32}4\cdot2^{-\log_32}\approx2.27,$$
and
$$\Mdl\C_2^{\frac13}=2\gamma^{-1}_{1-\log_32}\log_{\frac32}3\left(\log_4\frac 32\right)^{\log_32}\!\approx 2.19.$$

An older proof of the following proposition can be found in
\cite{jiangchen}, where the authors use the simpler formula for (both
upper and lower) Minkowski contents omitting the normalising factor.

\begin{proposition}
  The set $\C_N^r\subseteq\R$ is Minkowski non-degenerate with box
  counting dimension $$\dim_B\C_N^r\equiv -\frac{\log N}{\log r}=d.$$
  The upper Minkowski content at the dimension $d$ is
  $$\Mdu\C_N^r=2N\left(\frac s2\right)^d\frac{1-r}{1-Nr}\gamma_{1-d}^{-1}$$
  and the corresponding lower Minkowski content is
  $$\Mdl\C_N^r=\frac2{1-d}\left(\frac{1-d}{2d}\right)^d\gamma_{1-d}^{-1}.$$
\end{proposition}

Again, we can use $\Phi_1$ to identify $\C_N^r$ with its image on
$S^1$, and so we have
$$\dim_B\F_\alpha\C_N^r = \frac{1-\log_rN}{1+\alpha}$$ and
$$\dim_B\O_\alpha\C_N^r = \frac{1}{1+\alpha}-\log_rN.$$

\begin{figure}[t]
  \centering
  \parindent0pt
  \begin{subfigure}[b]{0.24\textwidth}
    \includegraphics[width=\textwidth]{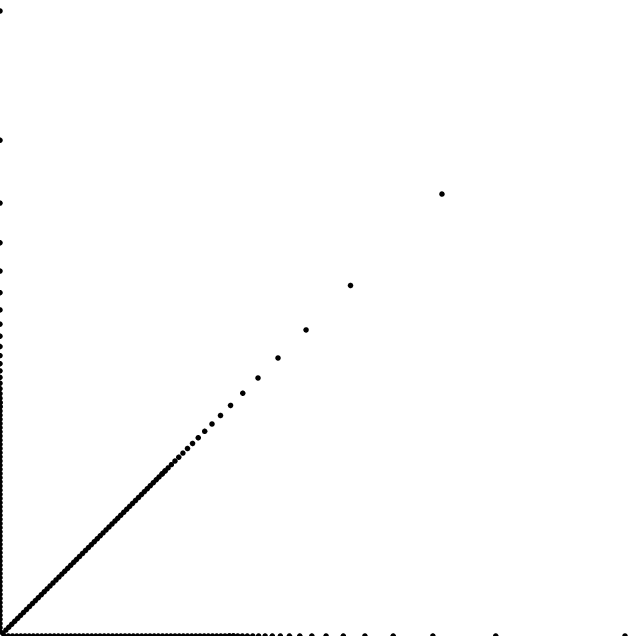}
    \caption{$\alpha\approx0.33$}
  \end{subfigure}
  \begin{subfigure}[b]{0.24\textwidth}
    \includegraphics[width=\textwidth]{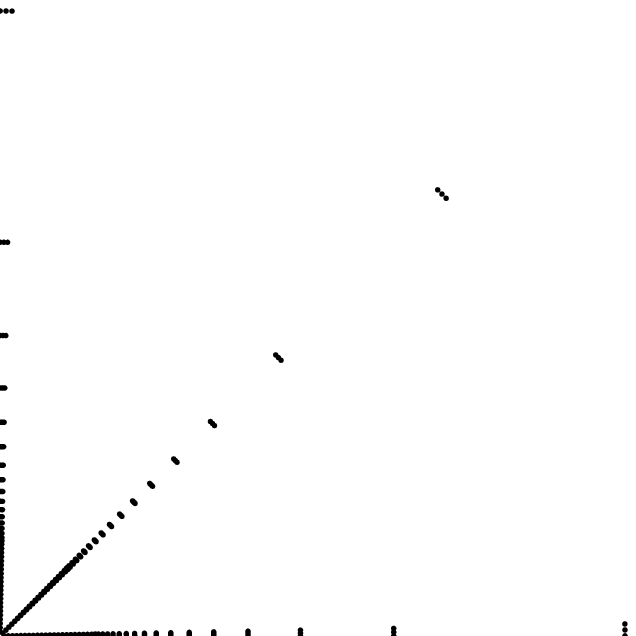}
    \caption{$\alpha\approx0.67$}
  \end{subfigure}
  \begin{subfigure}[b]{0.24\textwidth}
    \includegraphics[width=\textwidth]{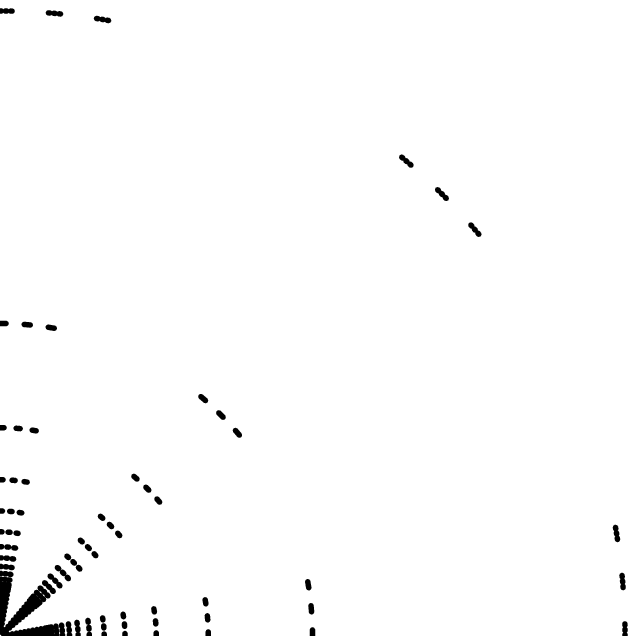}
    \caption{$\alpha\approx1.00$}
  \end{subfigure}
  \begin{subfigure}[b]{0.24\textwidth}
    \includegraphics[width=\textwidth]{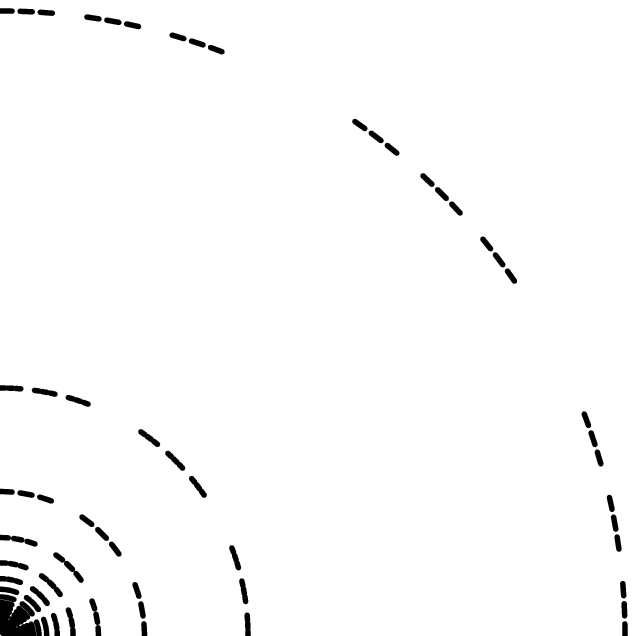}
    \caption{$\alpha\approx1.33$}
  \end{subfigure}
  \caption{$\F_\alpha\C_3^r$ nests of dimension $3/4$.}
  \label{fig:CF}
\end{figure}

\begin{figure}[t]
  \centering
  \parindent0pt
  \begin{subfigure}[b]{0.24\textwidth}
    \includegraphics[width=\textwidth]{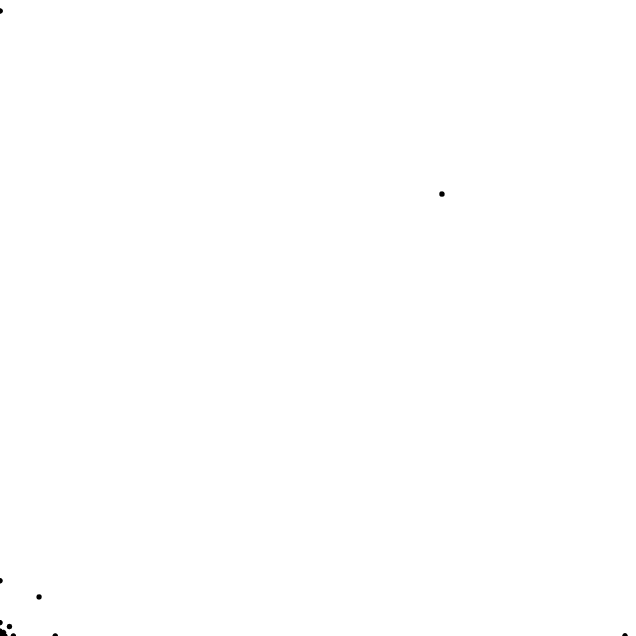}
    \caption{$\dim_B\F_\alpha\C_3^r\approx 0.25$}
  \end{subfigure}
  \begin{subfigure}[b]{0.24\textwidth}
    \includegraphics[width=\textwidth]{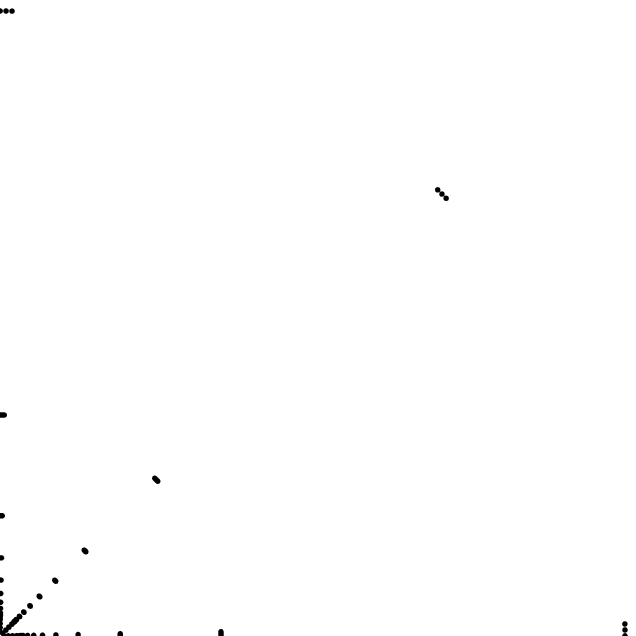}
    \caption{$\dim_B\F_\alpha\C_3^r\approx 0.50$}
  \end{subfigure}
  \begin{subfigure}[b]{0.24\textwidth}
    \includegraphics[width=\textwidth]{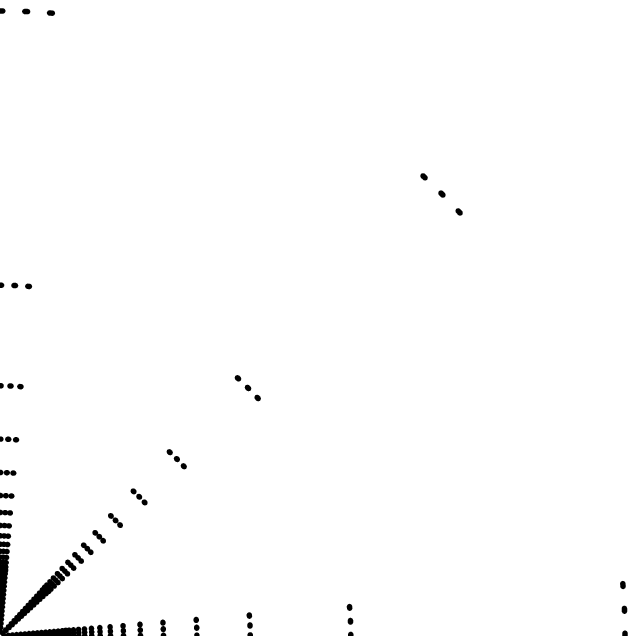}
    \caption{$\dim_B\F_\alpha\C_3^r\approx 0.75$}
  \end{subfigure}
  \begin{subfigure}[b]{0.24\textwidth}
    \includegraphics[width=\textwidth]{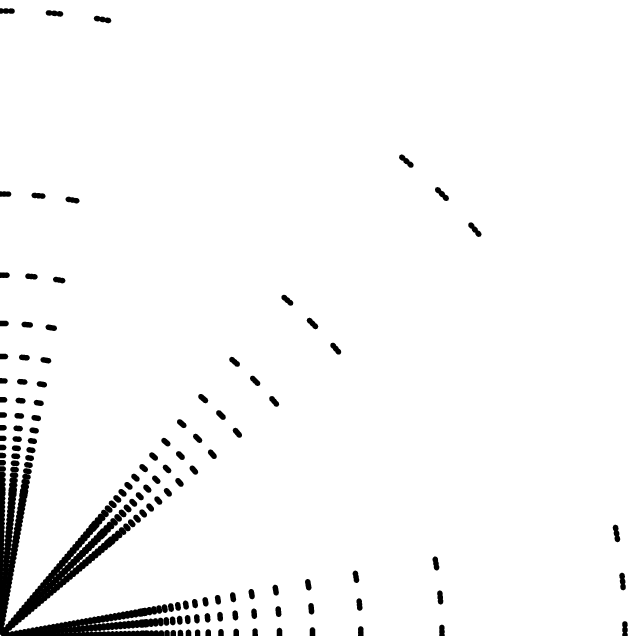}
    \caption{$\dim_B\F_\alpha\C_3^r\approx 1.00$}
  \end{subfigure}
  \caption{$\F_\alpha\C_3^r$ nests.}
  \label{fig:CV}
\end{figure}

\subsection{Numerical verification of the results}

Since the main results of this article are asymptotic in nature, there
is always a risk that we will not be able to reproduce such results
in numerical computations. Luckily, as can be seen in figures
\ref{fig:experiment-fixed} and \ref{fig:experiment-variable} we can
observe the dimensions to a reasonable accuracy.

We used the same algorithm for producing figures \ref{fig:BF},
\ref{fig:BV}, \ref{fig:CF} and \ref{fig:CV} and also for the explicit
computation of the dimensions.

We use the same technique used in the proofs of the main results, in
particular we use of Lemma \ref{emoneemtwo} for producing numbers
$m_1(\epsilon)$ (the number of isolated components, the ``tail'' of
the fractal) and $m_2(\epsilon)$, the number of $2\epsilon$-separated
elements that cover the ``core'' of the fractal nest.

For figures \ref{fig:BF}, \ref{fig:BV}, \ref{fig:CF} and \ref{fig:CV}
we used a program written in the Python programming language (version
3.6) that outputs an encapsulated PostScript (eps) description of the
fractal. PostScript is well-suited as a page description language
since it allows for global scaling and setting of line width using the
{\tt setlinewidth} command that is defined in the standard as ``up to
two pixels'' \cite[pg.\ 674]{postscript} best approximation of the
Minkowski sausage of half radius of the given parameter when rendered.

The figures themselves have the half of the line width parameter
$\epsilon$ set at $1/300$ of the height and width of the picture.

For $(\alpha,\beta)$-bi-fractals, we first obtain the radius of the
element of the nest using $r^{(1)}_k=k^{-\alpha}$ for
$k\in\{1,\ldots, m_1(\epsilon)\}$ and $r^{(2)}_k=2k\epsilon$ for
$k\in\{1,\ldots, m_2(\epsilon)\}$. Then, at each radius $r_k^{(i)}$ we
draw the the set $D_\beta$ by the same construction, using
$m_1\big(\frac{4\epsilon}{\pi r_k^{(i)}}\big)$ and
$m_2\big(\frac{4\epsilon}{\pi r_k^{(i)}}\big)$, both with respect
to the exponent $-\beta$ instead of $-\alpha$.

For uniform Cantor nests, we repeat the same initial part to obtain
$r_k^{(i)}$, we use a standard recursive algorithm for describing
segments of the Cantor set, where the number of iterations depends on
the radius of the nest element, since we are interested only in
segments that have gaps larger than $2\epsilon/r_k^{(i)}$.

In Figures \ref{fig:BF} and \ref{fig:CF} we show
$(\alpha, \beta)$-bi-fractals $\F_\alpha D_\beta$ and uniform Cantor
nests $\F_\alpha\C_3^r$ of fixed dimension $d=3/4$ and various
$\alpha$ with $\beta$ and $r$ computed from
\begin{align}
  r &= N^{-\frac{1}{\delta}}\label{r:num}\\
  \beta &= \frac1\delta - 1\label{beta:num}\\
  \delta &= d\alpha + d -1.\label{delta:num}
\end{align}

The constraint of the main result that $\alpha$ plays a role in the
total dimension only if $\alpha\delta<1$ limits $\alpha$ to
\begin{align}\alpha\in\left\langle{ \frac1d-1,
      \frac1d}\right\rangle\label{alphainterval}\end{align}

In Figures \ref{fig:BV} and \ref{fig:CV} we vary the total dimension
$d\in\langle\frac14,1\rangle$ and we
set \begin{align}\alpha=\frac1d-\frac12 \label{varying-alpha}\end{align}
so that $\alpha$ is always in the centre of the interval given in
(\ref{alphainterval}). Then we compute $\delta$ using
(\ref{delta:num}), and finally we produce the parameters $\beta$ and
$r$ from (\ref{beta:num}) and (\ref{r:num}).


\begin{figure*}[ht]
  \centering
  \includegraphics[width=\textwidth]{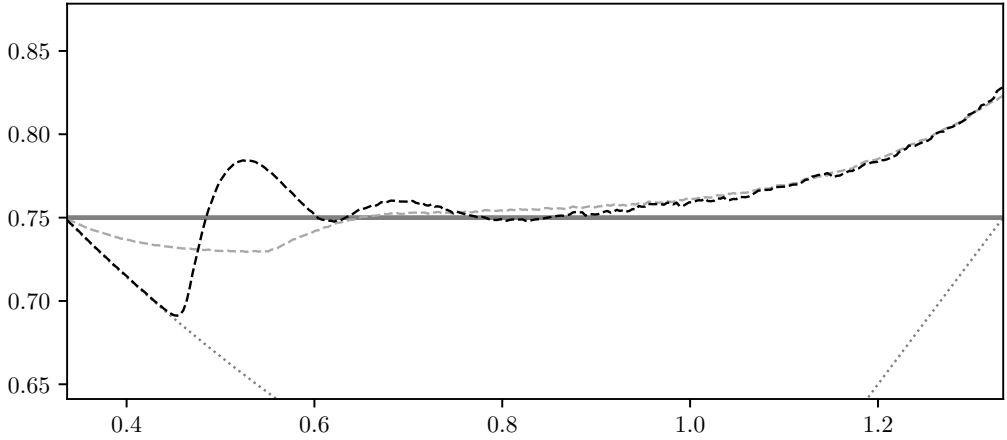}
  \caption{Relative error for numerical computation of the dimensions of
    $\F_\alpha D_\beta$ (grey dashed line) and $\F_\alpha\C_3^r$
    (black dashed line) for fixed dimension $3/4$ (bold horizontal line).}
  \label{fig:experiment-fixed}
\end{figure*}

For Figure \ref{fig:experiment-fixed} we fix the dimension of the whole
fractal nest to be $3/4$ and we vary the parameter $\alpha$.

We compute the parameters $\beta$ for the $(\alpha,\beta)$-bi-fractal
and $r$ for the $\F_\alpha\C_3^r$ Cantor nest. For ten different
values of $\epsilon$ ranging from $2^{-25}$ to $2^{-10}$ we count the
number $N_\epsilon$ of points necessary to draw the fractals. Using
Python's {\tt scikit-learn} library \cite{scikit}, we find the slope
for the linear regression of $\ln N_\epsilon$ against $-\ln\epsilon$.

As can be seen in Figure \ref{fig:experiment-fixed}, the results of
these computations come mostly within 10\% precision, with relatively
short execution time, finishing within minutes on a laptop
computer. The falling dotted line at the left side of Figure
\ref{fig:experiment-fixed}
represents $$\dim_BE_\alpha=\frac1{\alpha+1}$$ and the rising dotted
line on the right side is $\delta$, the dimension of the set copied by
the nest. The Cantor nests are sensitive in the beginning, following
the $(\alpha+1)^{-1}$ curve. On the right-hand side, as we approach
the Minkowski degenerate point when $\alpha=\delta^{-1}$ we should
expect the error to grow, as it does.


\begin{figure*}[ht]
  \centering
  \includegraphics[width=\textwidth]{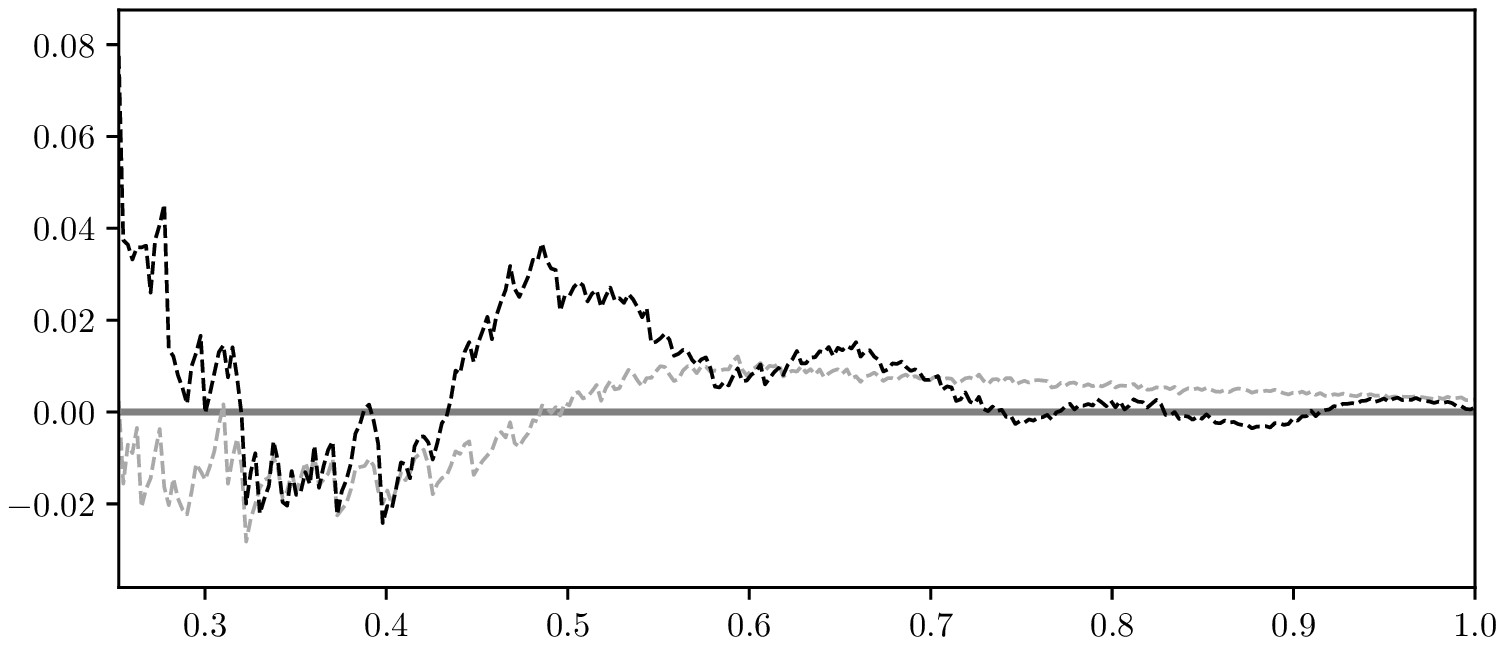}
  \caption{Numerical computation of the dimensions of
    $\F_\alpha D_\beta$ (grey dashed line) and $\F_\alpha\C_3^r$ (black dashed line)
    for varying dimensions.}
  \label{fig:experiment-variable}
\end{figure*}

In Figure \ref{fig:experiment-variable}, we display the relative error
of the same method against a varying total box counting dimension,
$d\in\left\langle0.25, 1\right\rangle$ as in figures \ref{fig:BV} and
\ref{fig:CV}, for $\alpha$ defined by (\ref{varying-alpha}).

Since $\alpha$ here is defined to be quite distant from $\delta^{-1}$
with $$\alpha\delta = \frac12 - \frac d4,$$ we expect the error to be
relatively low, with Figure \ref{fig:experiment-variable} showing the
error under 5\% for both types of nests.


\begin{figure*}[ht]
  \centering
  \includegraphics[width=\textwidth]{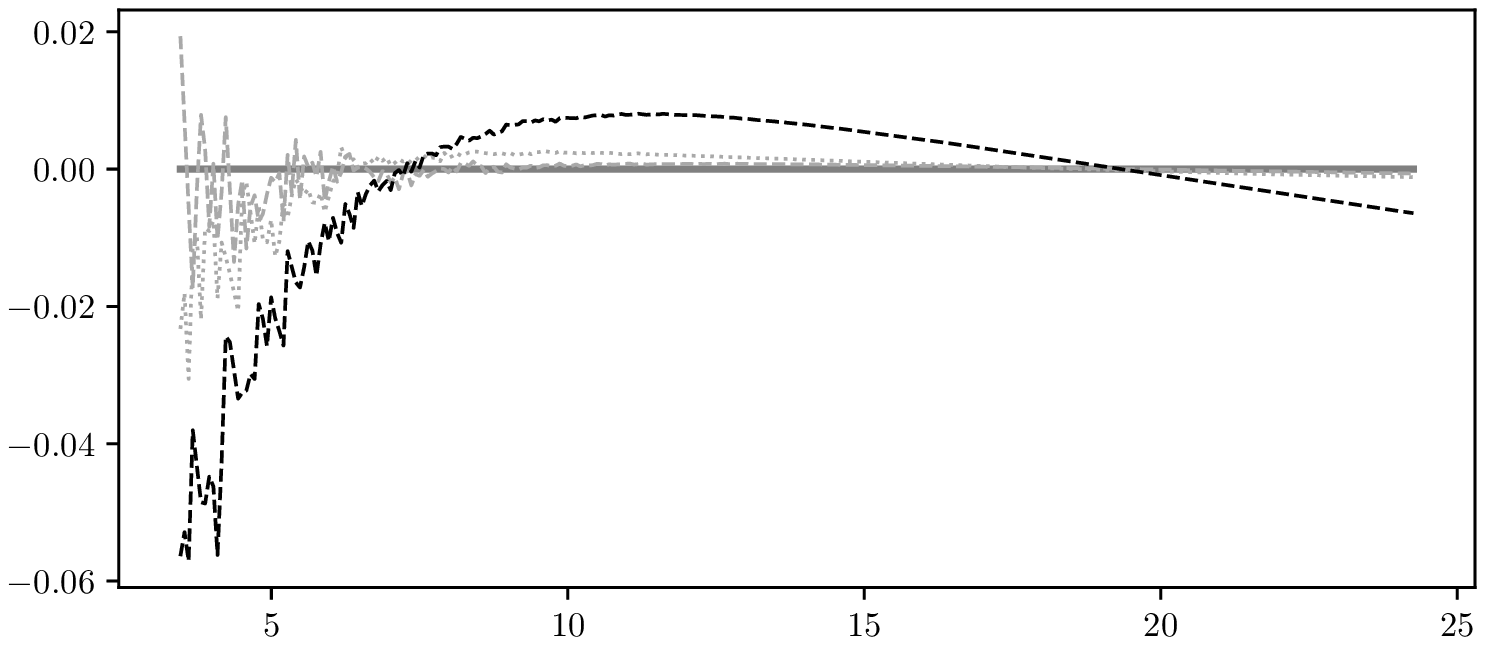}
  \caption{Relative deviations of $\ln N_\epsilon$ from linear
    regression with respect to $-\ln\epsilon$ for
    $(\alpha,\beta)$-bi-fractals of dimension $3/4$ with $\alpha=4/5$
    (dashed grey line), $\alpha=4/3$ (dashed black line) and
    $\alpha=3$ (dotted grey line).}
  \label{fig:experiment-individual-bifractal}
\end{figure*}


\begin{figure*}[ht]
  \centering
  \includegraphics[width=\textwidth]{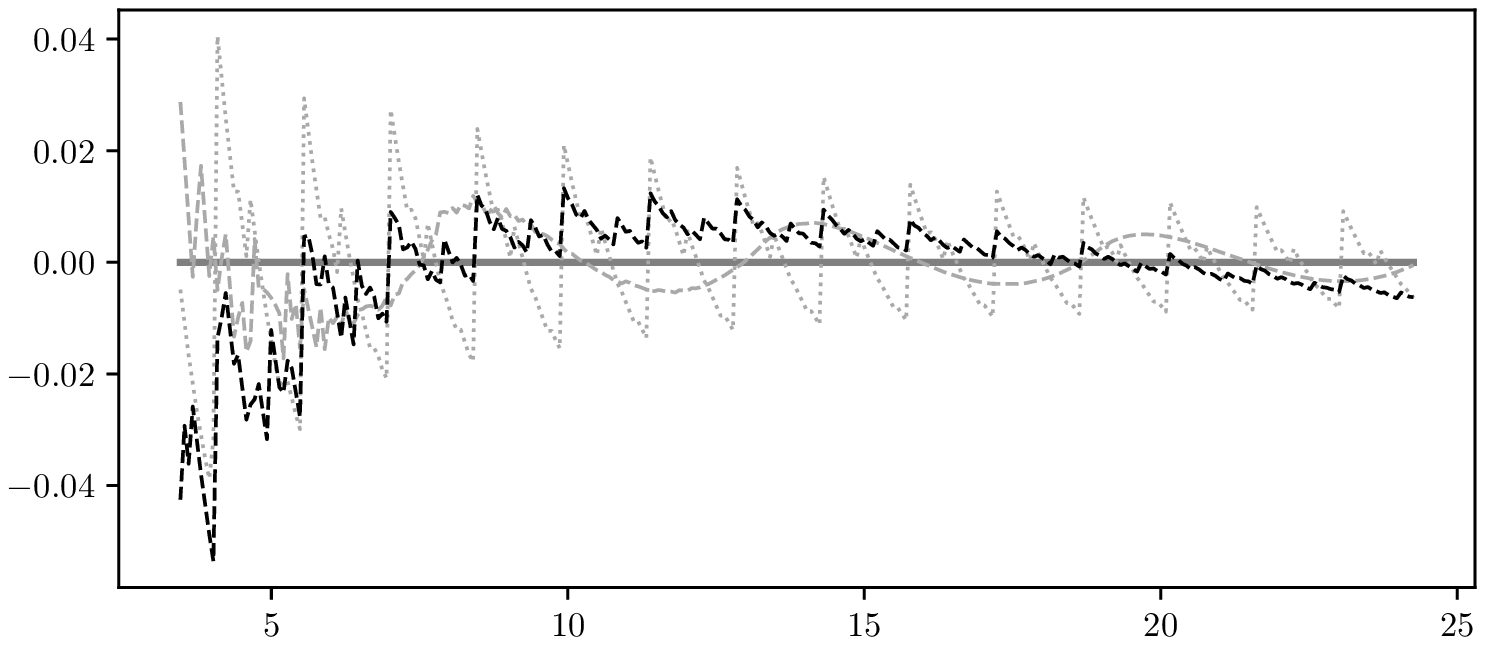}
  \caption{Relative deviations of $\ln N_\epsilon$ from linear
    regression with respect to $-\ln\epsilon$ for uniform Cantor nests
    of dimension $3/4$ with $\alpha=4/5$ (dashed grey line),
    $\alpha=4/3$ (dashed black line) and $\alpha=3$ (dotted grey
    line).}
  \label{fig:experiment-individulal-cantor}
\end{figure*}
To see the behaviour more clearly, in figures
\ref{fig:experiment-individual-bifractal} and
\ref{fig:experiment-individulal-cantor} we show the relative
deviations from linear regression for more detailed samples, with 300
samples for $\epsilon$ between $2^{-5}$ and $2^{-35}$. The dimension
of the set is fixed at $3/4$.

For $\alpha=4/5$ and $\alpha=3$ in both figures the approximation of
the dimension is very good, the numerically obtained dimension is
between $0.751$ and $0.754$, while at the critical point (where we
lose Minkowski-non-degeneracy, $\alpha=4/3$), the approximate
dimension is around $0.82$. At the critical point, the relative error
shows the most bias in both figures.

In figure \ref{fig:experiment-individulal-cantor} we can also observe
the effects of uniform Cantor sets having distinct upper and lower
Minkowski content, with visible oscillations of content at different
scales.

\section{Proofs of main results}

In this section we prove the main results. After the introduction of
the asymptotic notation we use and the lemmas necessary for the proof
of theorems \ref{th:inner} and \ref{th:outer}.

\subsection{Asymptotic notation}

In studying box counting dimensions, asymptotic notation is often
useful. Here we opt for $\sim$-style notation, like in \cite{tricot}
for mutually bounded sequences and functions, corresponding to
$\Theta$ in classical big-Oh notation \cite{knuth:ooO}.

\begin{definition}[Sequence and function equivalence]
  Let $a_n$ and $b_n$ be two positive sequences. We say that $a_k$
  and $b_k$ are {\bf equivalent} and denote it by $$a_k\sim_k b_k$$ if there exists a
  number $M\geq1$ such that  for all
  $k\in\N$,  $$M^{-1}a_k\leq b_k \leq Ma_k.$$

  Similarly, let $I$ be a set and let $f,g\colon I\to\R^+$. We say that
  $f$ and $g$ are {\bf equivalent on $I$}, denoting $$f(x)\sim_{x\in I} g(x)$$ if there
  exists a constant $M\geq 1$ such that for every $x\in I$ we have
  $$M^{-1}g(x)\leq f(x)\leq Mg(x).$$
\end{definition}

When the domain of equivalence is unambiguous, we will use only the
symbol $\sim$ to denote the equivalence.

\subsection{The lemmas and the proofs}

\begin{lemma}\label{composition}
  Let $f,g\colon I\to\R$ and let $\phi\colon J\times\N \to I$ and $m\colon J\to\N$. If
  $$f(x)\sim_{x\in I} g(x)$$ then
  $$\sum_{k=1}^{m(x)}(g\circ\phi)(x,k)\sim_{x\in J} \sum_{k=1}^{m(x)}(h\circ\phi)(x,k).$$
\end{lemma}

\begin{proof} Let $x\in J$ and $k\in\N$. Since $f\sim_Ig$ we have
  $$ M^{-1}(g\circ \phi)(x,k)\leq (f\circ\phi)(x,k)\leq M(g\circ
  \phi)(x,k).$$ for some $M\geq 1$. Taking the sum of parts of the inequality to $m(x)$,
  we get
  \begin{align*}
    M^{-1}\sum_{k=1}^{m(x)}(g\circ \phi)(x,k)\leq{}& \sum_{k=1}^{m(x)}(f\circ\phi)(x,k)
    \\&\leq M\sum_{k=1}^{m(x)}(g\circ \phi)(x,k),
  \end{align*}
  proving the lemma.
\end{proof}

The following Minkowski non-degeneracy condition is a useful
intuition on what the box-counting dimension represents, namely that
the ambient area of the Minkowski sausage is asymptotically equivalent
to radius to the power of the ``complementary'' dimension of the ambient
space.

\begin{lemma}\label{ndgprop}
  Let $S\subseteq \R^n$ be such that $\dim_BS\equiv d$ and let
  $L>0$. Then that for $\epsilon\langle0,L]$ we
  have $$\lambda(S)_\epsilon\sim \epsilon^{n-d}.$$
\end{lemma}

\begin{proof}
  This is an direct consequence of $A_\epsilon^{n,d}$ being
  bounded from both above and below as $\epsilon\to 0$ and the fact that
  $\lambda(S)_\epsilon$ is continuous and monotonous as a function of
  $\epsilon$.
\end{proof}

\begin{lemma}\label{geomlemma} Let $S\subseteq\R^n$ be a Borel set and let $x, \epsilon>0$. Then,
  $$((x)S)_\epsilon = (x)(S)_{\frac{\epsilon}{x}}$$ and, consequently,
  $$\lambda((x)S)_\epsilon = x^n\lambda(S)_{\frac{\epsilon}{x}}.$$
\end{lemma}

\begin{proof} Suppose $a\in ((x)S)_\epsilon$. This is true if and only
  if $$\inf_{b\in S}\|a-xb\|\leq \epsilon.$$ Since $x>0$, we
  get $$\inf_{b\in S}\left\|{\frac1x}a-b\right\|\leq \frac\epsilon x,$$ so
  $a\in ((x)S)_\epsilon$ if and only if
  $x^{-1}a\in (S)_{\frac \epsilon x}$, which is equivalent to
  $a\in (x)(S)_{\frac \epsilon x}$, proving the lemma.
\end{proof}

\begin{corollary}\label{iksnaentu}
  Under the same assumptions, with $x\geq1$ ($0<x\leq1$) we
  have
  \begin{align*}
    \lambda(S)_{x\epsilon}&\leq x^n\lambda(S)_\epsilon\\
    \bigg(\lambda(S)_{x\epsilon}&\geq x^n\lambda(S)_\epsilon\bigg).
  \end{align*}
\end{corollary}

\begin{lemma}\label{asymcore} Let $S\subseteq \R^n$ be a Borel set and
  let $\epsilon>0$. For a sequence $a_k\sim_k 1$,
  $$\sum_{k=1}^m\lambda(S)_{\epsilon a_k}\sim_m m\lambda(S)_\epsilon.$$
\end{lemma}

\begin{proof}
  Since $M^{-1}\leq a_k\leq M$ for some $M\geq 1$, we have
  that
  $$\lambda(S)_{\epsilon a_k}\leq \lambda(S)_{\epsilon M} \leq
  M^n\lambda(S)_\epsilon,$$ and likewise for the lower bound, we
  have $$\lambda(S)_{\epsilon a_k}\sim_k\lambda(S)_\epsilon.$$
  Taking the sum of both sides from $1$ to $m$, we prove the lemma.
\end{proof}

\begin{lemma}[Dense covering lemma]\label{coverlemma}Let $S\subseteq \R^n$ be a Borel set.
  Suppose that for every $\epsilon>0$ there is a set
  $I_\epsilon\subseteq\R^+$ and a (possibly infinite) sequence
  $a_k^{\epsilon}$,  such that
  $$I_\epsilon\subseteq (A^\epsilon)_\epsilon \subseteq (I_\epsilon)_\epsilon.$$
  where $m(\epsilon)$ is the (possibly infinite) length of
  $a_k^{(\epsilon)}$ and $A^\epsilon=\{a_k\,|\,k\leq m(\epsilon)\}$.
  Then,
  $$\lambda\left(\bigcup_{x \in I_\epsilon}(x)S\right)_\epsilon
  \sim\lambda\left(\bigcup_{k=0}^{m(\epsilon)}(a_k^{(\epsilon)})S\right)_\epsilon,$$
  with respect to $\epsilon$.
\end{lemma}

\begin{proof}

  The inclusions
  $$I_\epsilon\subseteq (A^\epsilon)_\epsilon \subseteq (I_\epsilon)_\epsilon$$
  imply, by triangle inequality,
  $$ (A^\epsilon)_\epsilon\subseteq
  (I_\epsilon)_\epsilon\subseteq (A^\epsilon)_{2\epsilon}.$$

  Hence,
  $$\left(\bigcup_{x \in I_\epsilon}(x)S\right)_\epsilon\subseteq
  \left(\bigcup_{k=0}^{m(\epsilon)}(a_k^{(\epsilon)})S\right)_{2\epsilon},$$
  and by monotonicity of the Lebesgue measure, and applying Corollary
  \ref{iksnaentu} we have
  $$\lambda\left(\bigcup_{x \in I_\epsilon}(x)S\right)_\epsilon
  \leq
  2^n\lambda\left(\bigcup_{k=0}^{m(\epsilon)}(a_k^{(\epsilon)})S\right)_\epsilon$$

  By the same argument applied to
  $(A^\epsilon)_\epsilon \subseteq (I_\epsilon)_\epsilon$,
  we have
  $$
  \lambda\left(\bigcup_{k=0}^{m(\epsilon)}(a_k^{(\epsilon)})S\right)_\epsilon
  \leq
  2^n\lambda\left(\bigcup_{x \in I_\epsilon}(x)S\right)_\epsilon,
  $$
  proving the lemma.
\end{proof}

\begin{lemma}\label{emoneemtwo}
  Let $\alpha>0$. For every $\epsilon>0$ there exist numbers
  $m_1(\epsilon)$ and $m_2(\epsilon)$ (denoted further by $m_1$ and
  $m_2$) satisfying:
  \begin{align*}
    (m_1+1)^{-\alpha}\!\!-(m_1+2)^{-\alpha}\!&< \!2\epsilon\!\leq m_1^{-\alpha}\!\!-(m_1+1)^{-\alpha}\\
    (m_1+1)^{-\alpha}&\leq 2m_2\epsilon \leq m_1^{-\alpha}
  \end{align*}
  such that $$m_1(\epsilon)\sim m_2(\epsilon)\sim\epsilon^{\frac{-1}{1+\alpha}}.$$
\end{lemma}

\begin{proof}
  The existence of $m_1$ and the required asymptotic behaviour of $m_1$
  follows from the fact that
  $$n^{-\alpha}-(n+1)^{-\alpha}\sim_n n^{-(\alpha+1)}.$$

  For $m_2$ we observe that $2\epsilon$ fits between $m_1^{-\alpha}$
  and $(m_1+1)^{-\alpha}$, so we can
  construct
  $$m_2(\epsilon)=\left\lceil\frac{(m_1(\epsilon)+1)^{-\alpha}}{2\epsilon}\right\rceil\leq\frac{(m_1+1)^{-\alpha}}{2\epsilon}+1.$$

  Suppose $2\epsilon m_2>m_1^{-\alpha}$. Then
  \begin{align*}
    m_1^{-\alpha}&<
                   2\epsilon\left(\frac{(m_1+1)^{-\alpha}}{2\epsilon}+1\right)\\
                 &=(m_1+1)^{-\alpha}+2\epsilon\\
                 &\leq m_1^{-\alpha}
  \end{align*}
  so we have arrived at a contradiction.
\end{proof}

Now we turn to the proofs of our main theorems.

\begin{proof}[Proof of Theorem \ref{th:inner}]
  Let $\epsilon>0$. We apply Lemma \ref{emoneemtwo} to obtain the
  functions $m_1(\epsilon)$ and $m_2(\epsilon)$.

  We define the $\epsilon$-tail, $\mo T_\epsilon(S)$ as the part of
  the nest consisting of $\epsilon$-isolated components,
  $$\mo T_\epsilon S = \bigcup_{k=1}^{m_1}\mo (k^{-\alpha})S,$$
  and we define the $\epsilon$-core, $\mo C_\epsilon(S)$ as the remaining components
  $$\mo C_\epsilon S = \bigcup_{k>m_1}\mo (k^{-\alpha})S.$$

  Now, we have that
  $$\lambda\left(\F_\alpha S\right)_\epsilon= \lambda\left(\mo
    T_\epsilon S\right)_\epsilon+\lambda\left(\mo C_\epsilon
    S\right)_\epsilon.$$

  By construction, $2\epsilon m_2(\epsilon)$ is an upper bound on the Hausdorff distance of
  the limit set $\{0\}$ of the whole nest and the $\epsilon$-core.

  First, we find the asymptotic behaviour of the core by computing
  \begin{align}
    \lambda\left(\mo C_\epsilon S\right)_\epsilon
    ={}&\lambda\left(\bigcup_{k>m_1}(k^{-\alpha})S\right)_\epsilon\notag\\
        &\sim{}\lambda\left(\bigcup_{k=0}^{m_2}(2k\epsilon)S\right)_\epsilon\label{cn:c1}\\
       &=\sum_{k=0}^{m_2}\lambda\left((2k\epsilon)S\right)_\epsilon\notag\\
       &=\sum_{k=0}^{m_2}(2k\epsilon)^n\lambda(S)_{\frac1{2k}}\label{cn:c2}\\
    \sim{}&\epsilon^n\sum_{k=0}^{m_2}k^n\cdot k^{\delta-n}
            \sim{}\epsilon^n\sum_{k=0}^{m_2}k^\delta\label{cn:c3}\\
    \sim{}&\epsilon^nm_2^{1+\delta}
            \sim{}\epsilon^{n-\frac{\delta+1}{\alpha+1}},\notag
  \end{align}
  applying Lemma \ref{coverlemma} twice at step (\ref{cn:c1}), first
  time to the infinite sequence of $k^{-\alpha}$ for $k>m_1$ and then
  to the finite sequence $2\epsilon k$. At step (\ref{cn:c2}) we apply
  Lemma \ref{geomlemma} and to obtain (\ref{cn:c3}) we apply
  Proposition \ref{ndgprop} and Lemma \ref{composition}.

  For the tail, we have
  \begin{align}
    \lambda\left(\mo T_\epsilon S\right)_\epsilon
    &= \sum_{k=1}^{m_1}\lambda\left((k^{-\alpha})S\right)_\epsilon\notag\\
    &= \sum_{k=1}^{m_1}k^{-n\alpha}\lambda(S)_{k^\alpha\epsilon}\label{cn:t1}\\
    &\sim \sum_{k=1}^{m_1}k^{-n\alpha}\left(k^\alpha\epsilon\right)^{n-\delta}\notag\\
     &\sim \epsilon^{n-\delta}\sum_{k=1}^{m_1}k^{-\alpha\delta},\label{cn:t2}
  \end{align}
  where we apply Lemmas \ref{composition} and \ref{geomlemma} at step
  (\ref{cn:t1}) and Proposition \ref{ndgprop} at step (\ref{cn:t2}).

  Now, if $\alpha\delta<1$, we have
  $$ \lambda\left(\mo T_\epsilon S\right)_\epsilon\sim
  \epsilon^{n-\delta}m_1^{1-\alpha\delta}\sim
  \epsilon^{n-\frac{\delta+1}{\alpha+1}},$$ in which case we
  have $$\dim_BF_\alpha(S)\equiv \frac{\delta+1}{\alpha+1}.$$

  If $\alpha\delta>1$, we set $\beta=\alpha\delta-1$, and therefore
  the series $\sum_{k=1}^{\infty}k^{-\alpha\delta}$ converges, so we have
  $$ \lambda\left(\mo T_\epsilon S\right)_\epsilon\sim\epsilon^{n-\delta},$$ and so the total area is
  $$\lambda\left(\F_\alpha S\right)_\epsilon
  \sim \epsilon^{n-\delta}+\epsilon^{n-\frac{\delta+1}{\alpha+1}}=
  \epsilon^{n-\delta}(1+\epsilon^\frac{\beta}{\alpha+1}),$$
  proving $$\dim_B F_\alpha(S)\equiv \delta.$$

  In the case $\alpha\delta=1$ we
  have $$\frac{\delta+1}{\alpha+1}=\delta$$ so for the total area we have
  $$\lambda\left(\F_\alpha S\right)_\epsilon\sim\epsilon^{n-\delta}(1-\ln\epsilon).$$

  For any $\zeta>0$ we have
  $$\frac{\lambda\left(\F_\alpha S\right)_\epsilon}{\epsilon^{n-(\delta+\zeta)}}
  \sim \epsilon^{\zeta}(1-\ln\epsilon)\to 0$$ for the upper dimension and
  $$\frac{\lambda\left(\F_\alpha S\right)_\epsilon}{\epsilon^{n-(\delta-\zeta)}}
  \sim \epsilon^{-\zeta}(1-\ln\epsilon)\to+\infty,$$
  for the lower dimension, proving $$\dim_BF_{\frac1\delta}S=\delta,$$ a Minkowski-degenerate case.
\end{proof}

\begin{proof}[Proof of Theorem \ref{th:outer}]
  Let $\epsilon>0$. Again, we introduce $m_1$ and $m_2$ using Lemma
  \ref{emoneemtwo}.  Also, as in the previous proof, we define the
  $\epsilon$-tail and $\epsilon$-core; the $\epsilon$-tail,
  $\mo T_\epsilon S$ as the part of the nest consisting of
  $\epsilon$-isolated components, and the $\epsilon$-core,
  $\mo C_\epsilon(S)$ as the remaining components.

  Now, we have that
  $$\lambda\left(\O_\alpha S\right)_\epsilon= \lambda\left(\mo
    T_\epsilon S\right)_\epsilon+\lambda\left(\mo C_\epsilon
    S\right)_\epsilon.$$

  For the tail, we compute
  \begin{align}
    \lambda\left(\mo T_\epsilon S\right)_\epsilon
    &= \sum_{k=1}^{m_1}\lambda\left((1-k^{-\alpha})S\right)_\epsilon\label{on:c1}\\
    &= \sum_{k=1}^{m_1}\left(1-k^{-\alpha}\right)^n
      \lambda(S)_{\epsilon(1-k^{-\alpha})^{-1}}\label{on:c2}\\
    &\sim m_1\lambda(S)_\epsilon \sim \epsilon^{n-\delta-\frac1{\alpha+1}},
  \end{align}
  applying Lemma \ref{geomlemma} at step (\ref{on:c1}) and Lemmas
  \ref{composition} and \ref{asymcore} at step (\ref{on:c2}).

  For the core, we have
  \begin{align}
    \lambda\left(\mo C_\epsilon S\right)_\epsilon
    ={}&\lambda\left(\bigcup_{k>m_1}(1-k^{-\alpha})S\right)_\epsilon\label{on:t1}\\
    \sim{}&\lambda\left(\bigcup_{k=0}^{m_2}{(1-2k\epsilon})S\right)_\epsilon\label{on:t2}\\
       &=\sum_{k=0}^{m_2}(1-2k\epsilon)^n\lambda(S)_{\frac\epsilon{1-2k\epsilon}}\label{on:t3}\\
    \sim{}&m_2 \lambda(S)_\epsilon\sim\epsilon^{n-\delta-\frac1{\alpha+1}},\label{on:t4}
  \end{align}
  by applying Lemma \ref{coverlemma} at step (\ref{on:t1}), Lemma
  \ref{geomlemma} at step (\ref{on:t2}). To apply Lemma \ref{asymcore} at step
  (\ref{on:t3}) we note that
  $$1-m_1^{-\alpha}< 1-2k\epsilon\leq 1$$ from the defining condition on $m_2$ (Lemma \ref{emoneemtwo}) on
$m_2$. At step (\ref{on:t4}) we also apply Lemma \ref{composition}.

  Thus, we have shown that $$\dim_B\O_\alpha S\equiv \delta+\frac1{\alpha+1}.$$
\end{proof}

\section{Closing remarks and open problems}

In this article we have shown that an $\alpha$-regular fractal nests
based on a set $S$ of non-degenerate box counting dimension $\delta$,
has dimensions
$$\dim_B(\F_\alpha S) = \frac{\delta+1}{\alpha+1}\hbox{ or
}\dim_B(\F_\alpha S) = \delta$$ for nests of centre type, with
$\alpha\delta<1$ for the first case and $\alpha\delta\geq 1$ for the second, and we have shown
$$\dim_B(\O_\alpha S) = \delta+\frac1{\alpha+1}$$ for the outer type.

These results concur with examples given in \cite{lazura} for
hyper-spheres $\Sn1\subseteq\R^n$, where
\begin{align*}
  \dim_B\F_\alpha\Sn1&=\max\left\{n-1, \frac n{\alpha+1}\right\},\\
  \dim_B\O_\alpha\Sn1&=n -\frac \alpha{\alpha+1}.
\end{align*}

We have also shown that for sets $\F_\alpha S_\delta$ of dimension $d$
based on Minkowski non-degenerate sets $S_\delta$ of dimension
$\delta$ we have the following relations:
\begin{align*}
  \alpha&\in\left\langle{ \frac1d-1, \frac1d}\right\rangle,\\
  \delta&=d\alpha+d-1.
\end{align*}

These relationships allow us to study the efficacy of simpler
numerical techniques for fractal sets presented in this article.

It is well known \cite{falconer, federer, tricot} that the box
counting dimension is not continuous with respect to countable
unions. The exact behaviour in examples given here points to a subtler
structure that explains the formulas for fractal nests.

For the outer type of nests, this has already been discussed in
\cite[Remark 6]{zuzu}, but the behaviour of centre-type nests remains
less well understood. The formula for the centre-type nests obtained
here points to a multiplicative operation on the dimensions. Such
behaviour of dimensions is well known mostly in vector spaces for
tensor products. We propose here that there exists an abstract
tensor-like product $\otimes$ on Borel sets such that
$$\F_\alpha S\approx(S\times I)\otimes\F_\alpha\{1\}$$ with $\approx$
guaranteeing the existence of a bi-Lipschitz map between the sets.
We expect the following to hold,
$$\dim_B (A\otimes B) = \dim_BA\cdot\dim_BB,$$ along with $\otimes$ being
distributive with respect to Cartesian products (which behave
additively).

Perhaps such a product could be found through an analogy between
Lipschitz (or bi-Lipschitz) maps that don't increase or preserve the
box dimension, respectively, in uniform spaces \cite[Chapter
2]{bourbaki:GT} and linear maps in vector spaces, possibly using a kind
of a smash product \cite[pg.\ 435]{bredon:TG}.

\section*{Acknowledgements}

I would like to thank Vedran \v Ca\v ci\'c, Ida Dela\v c Marion, and
Irina Puci\'c for their helpful input and comments and my wife Marija
Gali\'c Mili\v ci\'c for her patience.

All of the code used to generate the figures in this article is available at \href{https://github.com/sinisa-milicic/nests1}{https://github.com/sinisa-milicic/nests1}.

\bibliographystyle{alpha}
\bibliography{nests1}

\end{document}